\pdfoutput=1
\documentclass[11pt]{article}
\usepackage[utf8]{inputenc}
\usepackage{geometry}

\usepackage{graphicx}
\graphicspath{ {./images/} }
\usepackage{amsmath} 
\usepackage{amssymb} 
\usepackage{listings}
\usepackage{verbatim}
\usepackage[dvipsnames]{xcolor}
\usepackage{color}
\usepackage[vcentermath]{youngtab}
\usepackage{amsthm}
\usepackage{tikz}
\usepackage{lscape}
\usepackage{hyperref}
\usepackage{authblk}
\usepackage[linesnumbered,lined,commentsnumbered,ruled]{algorithm2e}

\usepackage{my_preamble}

\usepackage[compact]{titlesec}
\titlespacing{\section}{0pt}{1ex}{1ex}
\def\ci{\perp\!\!\!\perp}
\hypersetup{
    colorlinks=true,
    linkcolor=magenta,
    filecolor=magenta,      
    urlcolor=magenta,
    citecolor=magenta
}
\title{Logarithmic Voronoi Cells for Gaussian Models}
\author[1]{Yulia Alexandr}
\author[2]{Serkan Ho\c{s}ten}
\affil[1]{University of California, Berkeley}
\affil[2]{San Francisco State University}
\date{}

\newcommand\mscriptsize[1]{\mbox{\scriptsize\ensuremath{#1}}}
\newcommand\mtiny[1]{\mbox{\tiny\ensuremath{#1}}}

\begin{document}
\maketitle

{\centering\footnotesize DEDICATED TO BERND STURMFELS ON HIS 60TH BIRTHDAY\par}
\begin{abstract}
We extend the theory of logarithmic Voronoi cells to Gaussian statistical models. In general, a logarithmic Voronoi cell at a point on a Gaussian model is a convex set contained in its log-normal spectrahedron. We show that for models of ML degree one and linear covariance models the two sets coincide. In particular, they are equal for both directed and undirected graphical models. We introduce decomposition theory of logarithmic Voronoi cells for the latter family. We also study covariance models, for which logarithmic Voronoi cells are, in general, strictly contained in log-normal spectrahedra. We give an explicit description of logarithmic Voronoi cells for the bivariate correlation model and show that they are semi-algebraic sets. Finally, we state a conjecture that logarithmic Voronoi cells for unrestricted correlation models are not semi-algebraic.  

\end{abstract}
\section{Introduction}\label{intro}
This paper extends the study of logarithmic Voronoi cells, first introduced in \cite{AH20logarithmic} for statistical models of discrete random variables, to Gaussian models. For any point $\Sigma$ on a Gaussian model, its logarithmic Voronoi cell is the fiber of the maximum likelihood estimator. In other words, it is the set of all sample covariance matrices which pick $\Sigma$ as the maximum likelihood estimate. The logarithmic Voronoi cell at $\Sigma$ is a convex set (Proposition \ref{prop:logVoronoi-convex})
and is contained in another convex set, namely, the log-normal spectrahedron (Proposition \ref{prop:logVoronoi-in-log-normal}). Compared to the logarithmic Voronoi cell, the log-normal spectrahedron at $\Sigma$ is a nicer set, and we catalogue instances when the two sets coincide. This includes Gaussian models whose maximum likelihood degree is one (Corollary \ref{cor:ML-degree-one}), such as Gaussian models on directed acyclic graphs (Theorem \ref{thm:directed-mldegree}), and  all linear concentration models (Proposition \ref{prop:concentration}). In the latter case, we prove a decomposition theorem
for the logarithmic Voronoi cells when the model is based on an undirected decomposable graph (Theorem \ref{thm:concentration-decomposition}). Lastly, we study in detail the logarithmic Voronoi cells of the bivariate correlation model. In this case, we explicitly determine the logarithmic Voronoi cells and exhibit that, in general, they are semi-algebraic sets
not equal to log-normal spectrahedra.   Finally, we state Conjecture \ref{conjecture} which claims that the logarithmic Voronoi cells of the elliptope are not semi-algebraic. Our work introduces interesting families of spectrahedra and other convex sets motivated by algebraic statistics that need to be studied further with the tools of real and convex algebraic geometry. 

The code used in our computations throughout this paper is available on GitHub. \footnote{\url{https://github.com/yuliaalexandr/gaussian-log-voronoi}}.

Before getting into details, we illustrate the main themes with an example.

\begin{example}\label{ex:CI-model}

Consider the model $\Theta$ that is given as the intersection of the algebraic variety
$$\left\{\Sigma=(\sigma_{ij}):\sigma_{13}=0,\; \sigma_{12}\sigma_{23}-\sigma_{22}\sigma_{13}=0\right\}
= \left\{\Sigma=(\sigma_{ij}):\sigma_{13}=0,\; \sigma_{12}\sigma_{23}= 0\right\}$$
with the cone $\PD_3$ of positive definite symmetric $3\times 3$ matrices. This is the conditional independence model given by $X_1 \ci X_3$ and $X_1 \ci X_3 \, | \, X_2$, and it is the union of two linear planes of dimension four. We may write
$$\Theta=\left\{\begin{pmatrix}t_1&0&0\\0&t_2&t_3\\0&t_3&t_4\end{pmatrix}\succ 0: t_i\in\RR\right\}\cup \left\{\begin{pmatrix}s_1&s_2&0\\s_2&s_3&0\\0&0&s_4\end{pmatrix}\succ 0: s_i\in\RR\right\}.$$
Let $\Theta_1$ and $\Theta_2$ denote the two components above, respectively. Given a matrix $\Sigma\in\Theta$, the set of sample covariance matrices $S\in\PD_3$ that have $\Sigma$ as their maximum likelihood estimate form the \textit{logarithmic Voronoi cell} at $\Sigma$. The set of all matrices $S\in \PD_3$ that have $\Sigma$ as a critical point while optimizing the log-likelihood function with respect to $S$ over $\Theta$ is the \textit{log-normal spectrahedron} at $\Sigma$. The log-normal spectrahedron at a general matrix $\Sigma\in\Theta_1\setminus \Theta_2$ is two-dimensional, parametrized as
$$\left\{\begin{pmatrix}t_1&x_1&x_2\\x_1&t_2&t_3\\x_2&t_3&t_4\end{pmatrix}\succ 0: x_1,x_2\in\RR\right\}.$$
This spectrahedron is a semi-algebraic set, defined by the two inequalities
$$-x_1^2 + t_1t_2>0\text{ and }
-t_2x_2^2 + 2t_3x_1x_2-t_4x_1^2-t_1t_3^2+t_1t_2t_4>0.$$
Since $\Sigma$ is assumed to be positive definite, for any choice of $t_i$, the log-normal spectrahedron at $\Sigma$ is an ellipse. By symmetry the same is true of any $\Sigma\in\Theta_2\setminus \Theta_1$. For a point $\Sigma=\diag(\sigma_1,\sigma_2,\sigma_3)\in\Theta_1\cap\Theta_2$, the log-normal spectrahedron is three-dimensional, given as
$$\left\{(x,y,z)\in\RR^3:\begin{pmatrix}\sigma_1& x &y \\ x&\sigma_2&0\\y&0&\sigma_3 \end{pmatrix}\succ 0\text{ and }\begin{pmatrix}\sigma_1& 0 &y \\ 0&\sigma_2&z\\y&z&\sigma_3 \end{pmatrix}\succ 0\right\}.$$

The maximum likelihood degree of $\Theta$ is two with one critical point in each linear component. Namely, for a general matrix $S=(s_{ij})\in\PD_3$, the two critical points on the model are $\Sigma_1\in\Theta_1$, given by $t_1=s_{11},t_2=s_{22},t_3=s_{23},t_4=s_{33}$, and $\Sigma_2\in\Theta_2$, given by $s_1=s_{11},s_2=s_{12},s_3=s_{22},s_4=s_{33}$. Now consider a general matrix $\Sigma\in\Theta_1\setminus \Theta_2$. The logarithmic Voronoi cell at $\Sigma$ is a subset of its log-normal ellipse, and it can be written as
\begin{align}\label{logarithmic Voronoi-cell-inside-ellipse}
    \left\{S=\begin{pmatrix}t_1&x_1&x_2\\x_1&t_2&t_3\\x_2&t_3&t_4\end{pmatrix}\succ 0: \ell_n(\Sigma, S)\geq \ell_n(\Sigma',S)\right\}
\end{align}
where $\ell_n$ is the log-likelihood function and $\Sigma'=\begin{pmatrix}t_1&x_1&0\\x_1&t_2&0\\0&0&t_4\end{pmatrix}$. Writing out the inequality in (\ref{logarithmic Voronoi-cell-inside-ellipse}), we find that it is equivalent to
\begin{align}\label{strip-inequality}
-t_3\sqrt{t_1/t_4}\leq x_1\leq t_3\sqrt{t_1/t_4}.
\end{align}
Thus, the logarithmic Voronoi cell at $\Sigma\in\Theta_1\setminus \Theta_2$ is the log-normal ellipse at $\Sigma$ intersected with the strip defined by (\ref{strip-inequality}). In particular, it is a semi-algebraic set. We plot the logarithmic Voronoi cell for $t_1=1,t_2=2,t_3=1,t_4=3$ in Figure \ref{ellipses} (on the left). Similarly, one checks that the logarithmic Voronoi cell at $\Sigma\in\Theta_2\setminus\Theta_1$ is the semi-algebraic set
$$\left\{S=\begin{pmatrix}s_1&s_2&y_1\\s_2&s_3&y_2\\y_1&y_2&s_4\end{pmatrix}\succ 0:-s_2\sqrt{s_4/s_1}<y_2<s_2\sqrt{s_4/s_1}\right\}.$$ We plot the logarithmic Voronoi cell for $s_1=2,s_2=1,s_3=3,s_4=4$ in Figure \ref{ellipses} (on the right). Thus, the logarithmic Voronoi cell at a general point of $\Theta$ is not equal to its log-normal ellipse.

\begin{figure}[ht]
    \centering
    \includegraphics[width=0.4\textwidth]{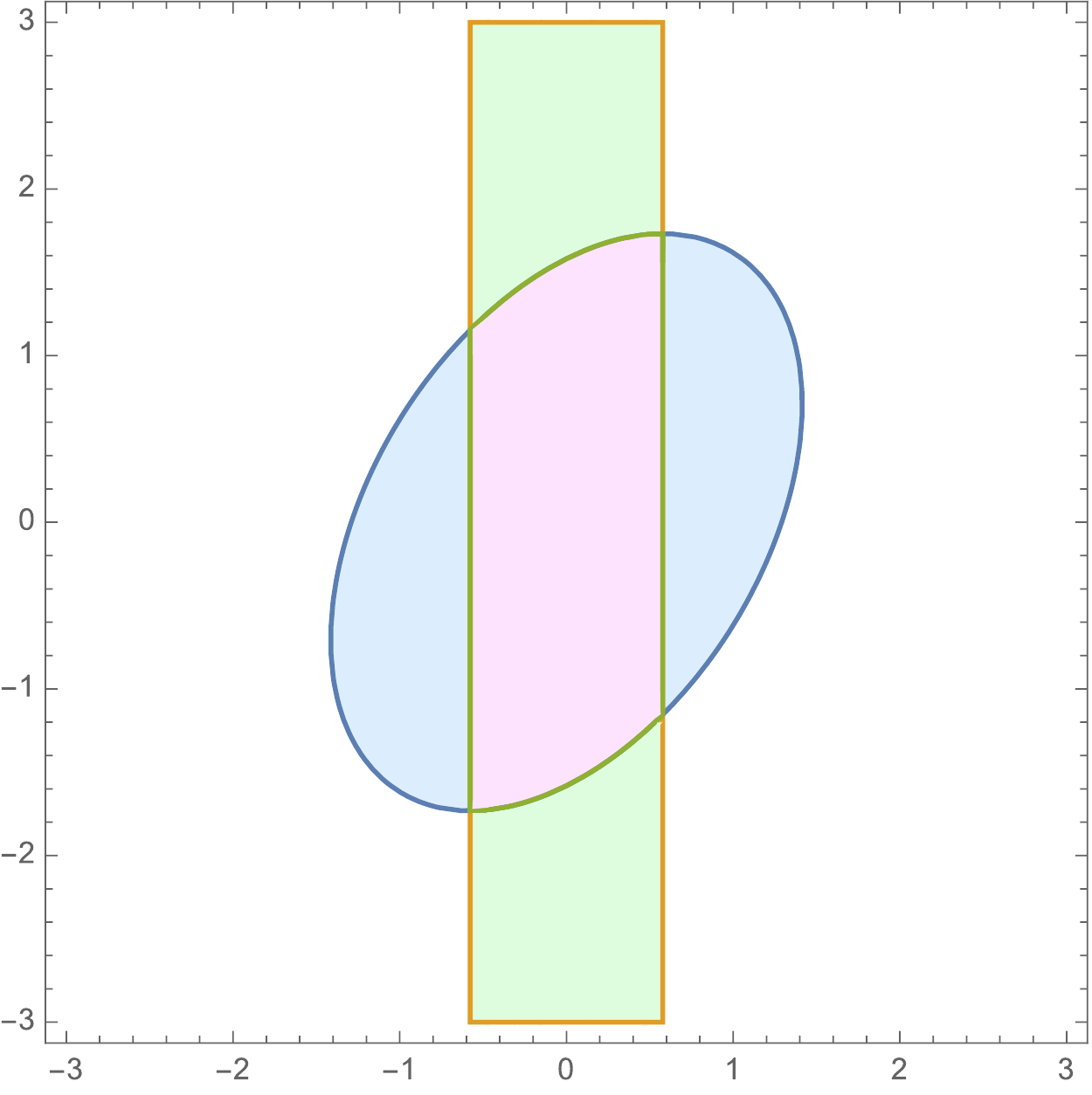} \includegraphics[width=0.4\textwidth]{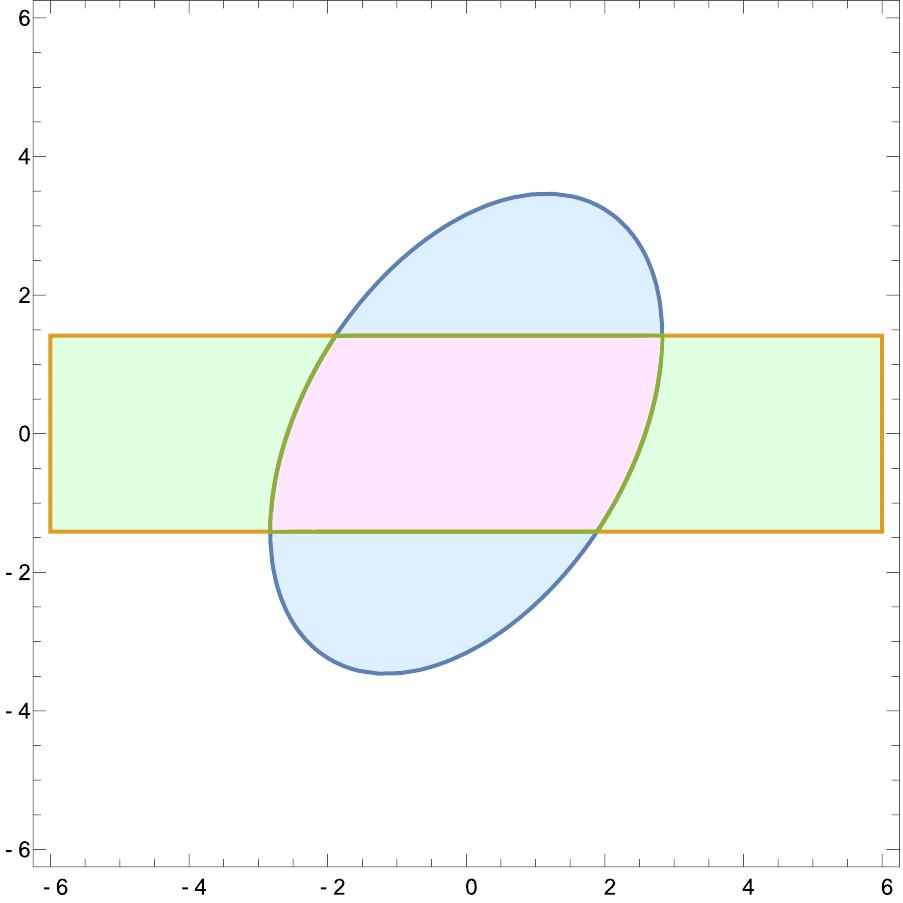} 
    \caption{Logarithmic Voronoi cells (in pink) of the model in Example \ref{ex:CI-model} plotted on the $(x_1,x_2)$-plane and $(y_1,y_2)$-plane, respectively.}
    \label{ellipses}
    
\end{figure}
\end{example}

We remark that despite strict containment of the logarithmic Voronoi cells in log-normal spectrahedra, 
the former is still a semi-algebraic set. This phenomenon is surprising, since the inequality
in (\ref{logarithmic Voronoi-cell-inside-ellipse}) that defines the logarithmic Voronoi cell together with the positive definiteness condition involves the log-likelihood function, which is not a polynomial function. We also note that at the singular points $\Sigma \in \Theta_1 \cap \Theta_2$ the logarithmic Voronoi cells equal the log-normal spectrahedra which are three-dimensional.

\section{Basics of Gaussian models and logarithmic Voronoi cells}

In this section, we give an introduction to Gaussian models and the maximum likelihood estimation problem for them. We define logarithmic Voronoi cells and log-normal spectrahedra
for such models and show that these two are equal when the maximum likelihood degree of the model is one. Our exposition follows \cite[Section 2.1]{DrtonSturmfelsSullivant2009LecturesOnAlgebraicStatistics}. 

Let $X=(\xi_1,\cdots, \xi_m)$ be an $m$-dimensional Gaussian random vector, which has the density function
$$p_{\mu,\Sigma}(x)=\frac{1}{(2\pi)^{m/2}(\det\Sigma)^{1/2}}\exp\left\{-\frac{1}{2}(x-\mu)^T\Sigma^{-1}(x-\mu)\right\},\;\;\; x\in\RR^m$$
with respect to the mean vector $\mu\in\RR^m$ and the covariance matrix $\Sigma\in\PD_m$, where $\PD_m$ is the cone of real symmetric positive definite $m\times m$ matrices. Such $X$ is said to be distributed according to the \textit{Gaussian distribution}, denoted by $\N(\mu, \Sigma)$.
For $\Theta\subseteq \RR^m\times \PD_m$, the statistical model 
$$\mathcal{P}_\Theta=\{\N(\mu, \Sigma): \theta=(\mu,\Sigma)\in\Theta\}$$
is called a \textit{Gaussian model}. Since the parameter space $\Theta$ completely determines the model, we will use $\Theta$ and $\mathcal{P}_\Theta$ interchangeably. For sampled data consisting of $n$ vectors $X^{(1)},\cdots, X^{(n)}\in\RR^m$, we define the sample mean and the sample covariance as
$$\bar{X}=\frac{1}{n}\sum_{i=1}^nX^{(i)}\hspace{1cm}\text{and}\hspace{1cm} S=\frac{1}{n}\sum_{i=1}^n(X^{(i)}-\bar{X})(X^{(i)}-\bar{X})^T,$$
respectively. Throughout this paper we fix a positive integer $n$. Given $n$ sampled data vectors, the log-likelihood function, up to an additive constant, is $$\ell_n(\mu,\Sigma)=-\frac{n}{2}\log\det\Sigma-\frac{n}{2}\tr\left(S\Sigma^{-1}\right)-\frac{n}{2}(\bar{X}-\mu)^T\Sigma^{-1}(\bar{X}-\mu).$$ For a fixed model $\Theta$, the sample
mean $\bar{X}$, and the sample covariance $S$, the \textit{maximum likelihood estimation} is the problem of finding the parameter pair $\hat{\theta}=(\hat{\mu}, \hat{\Sigma})\in \Theta$ at which the log-likelihood function $\ell_n$ is maximized. The \textit{maximum likelihood estimator} is the function $\Phi$ which maps the sample
data $(\bar{X}, S)$ to the maximizer of $\ell_n(\mu, \Sigma)$. For a point $\theta=(\mu,\Sigma)$ in the model, we define its \textit{logarithmic Voronoi cell} $\log\Vor_{\Theta}(\mu,\Sigma)$ to be the set of all $X^{(1)},\cdots, X^{(n)}\in\RR^m$ with sample mean $\bar{X}$ and sample covariance $S$ such that the log-likelihood function $\ell_n$ with respect to this sample is maximized at $\theta$. We will identify each sample $X^{(1)},\cdots, X^{(n)}$ with the tuple $(\bar{X},S)$ and consider any two samples whose sample mean and sample covariance are equal
to be the same. In this paper, we will study logarithmic Voronoi cells at only nonsingular points of Gaussian models. Hence, all our results are on this nondegenerate case, and we will 
not explicitly mention the nonsingularity of these points from now on. 

For any $U\subseteq\RR^m$ and $p\in U $, the \textit{Euclidean Voronoi cell} at $p$ is the set of all points in $\RR^m$ that are closer to $p$ than any other point in $U$ with respect to the Euclidean metric. Euclidean Voronoi cells of varieties were studied in \cite{CRSW2022} and are a topic 
in \textit{metric algebraic geometry} \cite{DiRoccoEklundWeinstein2020BottleneckDegreeOfAlgebraicVarieties, SturmfelsEuclideanDistanceDegree2016, maddiesthesis}. In general, logarithmic Voronoi cells are not equal to Euclidean Voronoi cells. However, it turns out they are the same for the next model.
\begin{prop}
Consider the Gaussian model with parameter space $\Theta=\Theta_1\times \{\text{Id}_m\}$ for some $\Theta_1\subseteq\RR^m$. For any point in this model, its logarithmic Voronoi cell is equal to its Euclidean Voronoi cell.
\end{prop}
\begin{proof}
We may identify the parameter space $\Theta$ with the subset of real vectors $\Theta_1\subseteq\RR^m$. For any sample $X^{(1)},\cdots, X^{(n)}$ with sample mean $\bar{X}$, the maximum likelihood estimate is the point in the model $\hat{\mu}\in\Theta_1$ that is closest to $\bar{X}$ in the Euclidean metric \cite[Prop. 2.1.10]{DrtonSturmfelsSullivant2009LecturesOnAlgebraicStatistics}. So, for any point $\mu\in\Theta_1$ in the model, the logarithmic Voronoi cell at $\mu$ is the set of all sample means $\bar{X}\in\RR^m$ that are closer to $\mu$ than any other point in $\RR^m$. This is precisely the Euclidean Voronoi cell at $\mu$.
\end{proof}

\begin{prop}
Let $\Theta=\RR^m\times\PD_m$ be the saturated Gaussian model. For any point in this model, its logarithmic Voronoi cell is the point itself.
\end{prop}
\begin{proof}
For any given sample $(\bar{X}, S)$, its maximum likelihood estimate $(\hat{\mu},\hat{\Sigma})$ is the point $(\bar{X}, S)$ itself \cite[Section 2.1] {DrtonSturmfelsSullivant2009LecturesOnAlgebraicStatistics}. Therefore, for any given point $(\mu, \Sigma)\in\Theta$, its logarithmic Voronoi cell is $\log\Vor_{\Theta}(\mu,\Sigma)=\{(\mu,\Sigma)\}$, as desired.
\end{proof}

Besides the above relatively simple cases, logarithmic Voronoi cells of Gaussian models are fairly complex convex sets. In the rest of the paper we will consider Gaussian models 
given by parameter spaces of the form $\Theta=\RR^m\times \Theta_2$ where $\Theta_2\subseteq \PD_m$. 
It is known that for any sample $(\bar{X},S)$, its maximum likelihood estimate is given by $(\hat{\mu},\hat{\Sigma})$ where $\hat{\mu}=\bar{X}$ and $\hat{\Sigma}$ is the maximizer of the log-likelihood function $\ell_n$ in the set $\Theta_2$.  In this case, we will refer to a Gaussian model as parametrized by $\Theta \subseteq \PD_m$ only, and therefore the log-likelihood function takes the form
$$\ell_n(\Sigma, S)=-\frac{n}{2}\log\det\Sigma-\frac{n}{2}\tr(S\Sigma^{-1}).$$ 
Hence, for any point $\Sigma\in\Theta$, its logarithmic Voronoi cell $\log \Vor_{\Theta}(\Sigma)$ is the set of all matrices $S\in \PD_m$ such that $\Sigma$ is a maximizer of $\ell_n(\Sigma, S)$, viewed as a function of $\Sigma$. 

\begin{prop} \label{prop:logVoronoi-convex} For a Gaussian model $\Theta \subseteq \PD_m$ and $\Sigma \in \Theta$, the logarithmic Voronoi cell $\log \Vor_{\Theta}{\Sigma}$ is a convex set. \end{prop}
\begin{proof}
The logarithmic Voronoi cell at $\Sigma$ is 
$$ \log\Vor_{\Theta}(\Sigma) \, = \, \{S \in \PD_m \, : \, \ell_n(\Sigma, S) \geq \ell_n(\Sigma', S) \mbox{ for all } \Sigma' \in \Theta \}.$$
Since $\ell_n(\Sigma, S)$ is linear in $S$, each inequality $\ell_n(\Sigma,S) \geq \ell_n(\Sigma',S)$ defines a closed halfspace. Therefore the logarithmic Voronoi cell at $\Sigma$
is the intersection of these halfspaces for each $\Sigma' \in \Theta$ and the convex
cone $\PD_m$.
\end{proof}

Now we introduce several definitions, which generalize the concepts introduced in \cite{AH20logarithmic} to Gaussian distributions. For a point $\Sigma\in \Theta$, we define the \textit{log-normal matrix space} at $\Sigma$, denoted by $\N_\Sigma\Theta$, to be the set of all symmetric $m\times m$ matrices $S$ such that $\Sigma$ appears as a critical point when optimizing $\ell_n(\Sigma, S)$. This is the set of all points such that the gradient $\nabla\ell_n(\Sigma,S)$ with respect to $\Sigma$ lies in the normal space of the model $\Theta$ at $\Sigma$. This condition is linear in $S$, so the log-normal matrix space is an affine linear space. Intersecting it with $\PD_m$, we obtain a spectrahedron $\mathcal{K}_{\Theta}(\Sigma) = \PD_m\cap \;\N_\Sigma\Theta$, which we call the \textit{log-normal spectrahedron} at $\Sigma$. We immediately obtain the following.
\begin{prop} \label{prop:logVoronoi-in-log-normal}
Each logarithmic Voronoi cell $\log\Vor_{\Theta}(\Sigma)$ is contained in the log-normal 
spectrahedron $\mathcal{K}_{\Theta}(\Sigma)$. In particular, 
\begin{align} \label{eq:difference}
\log\Vor_{\Theta}{\Sigma} = \{ S \in \mathcal{K}_{\Theta}{\Sigma}\, : \, \ell_n(\Sigma, S) \geq \ell_n(\Sigma', S) \mbox{ for all critical points } \Sigma'\}.
\end{align}
\end{prop}
The reverse of the containment above does not hold in general, as we have seen in Example \ref{ex:CI-model}. This is typical, and we will see more instances of this phenomenon.  
However, the two convex sets are equal if the log-likelihood function has a unique optimum on the model $\Theta$, and more strongly, if the maximum likelihood degree of the Gaussian model is one.

\begin{cor} \label{cor:unique-maximizer}
If $\ell_n(\Sigma,S)$ has a unique maximum $\Sigma$ over the Gaussian model $\Theta \subseteq \PD_m$, then $\log\Vor_{\Theta}(\Sigma) = \mathcal{K}_{\Theta}(\Sigma)$.
\end{cor}
\begin{proof}
Since $\Sigma$ is the unique maximum the inequalities in~(\ref{eq:difference}) are superfluous.
\end{proof}

\begin{definition}
The \textit{maximum likelihood degree} of a Gaussian model $\Theta \subseteq \PD_m$ is the number 
of nonsingular complex critical points of $\ell_n(\Sigma, S)$ for generic $S$ on the Zariski closure of $\Theta$ in the space of complex symmetric matrices.
\end{definition}

\begin{cor} \label{cor:ML-degree-one}
If the ML degree of a Gaussian model $\Theta \subseteq \PD_m$ is one then 
$\log\Vor_{\Theta}(\Sigma) = \mathcal{K}_{\Theta}(\Sigma)$ for every $\Sigma \in \Theta$.
\end{cor}
\begin{proof}
Since $\Sigma$ is the unique critical point, the result follows from Corollary \ref{cor:unique-maximizer}.
\end{proof}

\section{Linear concentration and undirected graphical models}
In a multivariate Gaussian distribution, the inverse of the covariance matrix $K = \Sigma^{-1}$
is known as the \textit{concentration matrix}. Linear concentration models \cite{anderson1970} are  given by concentration matrices which form a linear subspace. Let $\mathcal{L}$ be a 
$d$-dimensional linear
subspace of $m \times m$ real symmetric matrices. Then a \textit{linear concentration model} is given by
$$\Theta \, = \, \{ \Sigma \in \PD_m \, : \, K = \Sigma^{-1} \in \mathcal{L}\}. $$
The log-likelihood function  equals
$$ \frac{n}{2}\log\det K -\frac{n}{2}\tr(SK), $$ 
and it is a strictly concave function on $\mathcal{L} \cap \PD_m$. If $K_1, \ldots, K_d$ are
a basis of $\mathcal{L}$ and $S$ is a sample covariance matrix, the maximizer
${\hat \Sigma} = {\hat K}^{-1}$ of the log-likelihood function is the unique solution 
to 
\begin{align} \label{eq:concentration}
\tr( {\hat \Sigma}K_j) \, = \, \tr(SK_j), \quad j=1,\ldots, d.
\end{align}
This follows from writing $K = \sum_{i=1}^d \lambda_j K_j$ and taking partial derivatives of the log-likelihood function with respect to $\lambda_j$, $j=1, \ldots, d$; see \cite{SU10}. Therefore, we immediately get the following.
\begin{prop} \label{prop:concentration}
Let $\Theta$ be a linear concentration model given by $\mathcal{L} = \mathrm{span}\{K_1, \ldots, K_d\}$, and let $\Sigma \in \Theta$. Then
$$ \log\Vor_{\Theta}(\Sigma) = \mathcal{K}_{\Theta}(\Sigma) = \{ S \in \PD_m \, : \, \tr(SK_j) = \tr(\Sigma K_j), \,\, j=1, \ldots, d\}.$$
\end{prop}
\begin{proof}
The equality of the logarithmic Voronoi cell and the log-normal spectrahedron follows from
Corollary \ref{cor:unique-maximizer}. The linear description of these convex sets follows
from (\ref{eq:concentration}).
\end{proof}

\begin{cor} \label{cor:concentration-dim-one}
Let $\Theta$ be a one-dimensional linear concentration model spanned by $K \in \PD_m$.
For $\lambda >0$ and $\Sigma = \frac{1}{\lambda} K^{-1}$, the logarithmic Voronoi cell at $\Sigma$
is $\log\Vor_{\Theta}(\Sigma) = \{ S \in \PD_m \, : \, \tr(SK) = \frac{m}{\lambda} \}$. Therefore, it is the intersection of a translate of $\mathcal{L}^\perp$ with $\PD_m$ where 
$\mathcal{L} = \mathrm{span}(K)$.
\end{cor}
\begin{proof}
Since $\tr(\Sigma K) = \frac{m}{\lambda}$, the result follows Proposition \ref{prop:concentration}.
\end{proof}
\begin{cor} \label{cor:concentration-m=2}
When $m=2$, the logarithmic Voronoi cells of one-dimensional concentration models are convex regions defined by ellipses.
\end{cor}
\begin{proof}
Let $K = \begin{pmatrix} a & b \\ b & c\end{pmatrix} \succ 0$, $\lambda > 0$, and $\Sigma = \frac{1}{\lambda} K^{-1}$. 
With $S = \begin{pmatrix} s_{11} & s_{12} \\ s_{12} & s_{22} \end{pmatrix}$, from (\ref{eq:concentration}) we get $a s_{11} + 2b s_{12} + c s_{22} = \frac{2}{\lambda}$. Then
$\log\Vor_{\Theta}(\Sigma)$ is defined by the inequalities
$$ \frac{1}{a}\left(\frac{2}{\lambda} - 2b s_{12} - c s_{22}\right)s_{22} - s_{12}^2 \geq 0, \,\, \frac{1}{a}\left(\frac{2}{\lambda} - 2b s_{12} - c s_{22}\right) \geq 0, \,\, s_{22} \geq 0. $$
The quadric defines an ellipse since its $\Delta$-invariant and $\delta$-invariant \cite[Section 5.2]{Gibson1998}
are $\frac{\lambda^2}{a^2} \neq 0$ and $\frac{ac-b^2}{a^2} > 0$, respectively. Finally, the nonnegativity of the quadric implies the other inequalities for positive definite $S$.
\end{proof}

We would like to point out that, despite the concavity of $\ell_n(\Sigma,S)$ on a linear
concentration model $\Theta$, the maximum likelihood degree of $\Theta$ is much bigger than one. This was first studied in \cite{SU10} which included conjectures on the ML degree of such models. Most of these conjectures were settled in \cite{MMW21} and \cite{manivel2020complete}. See also \cite{AGKMS21} and \cite{JKW21} for related work. 

\subsection{Undirected graphical models}
Let $G=(V,E)$ be a simple undirected graph with $|V(G)|=m$. A \textit{concentration model} of $G$ is 
$$\Theta(G)=\{\Sigma\in\PD_m:(\Sigma)^{-1}_{ij}=0 \text{ if }ij\notin E(G)\text{ and }i\neq j\}.$$

Concentration models of undirected graphs are examples of linear concentration models. Thus, their logarithmic Voronoi cells are equal to the log-normal spectrahedra. Following Proposition \ref{prop:concentration}, we can describe logarithmic Voronoi cells explicitly as
$$\log\Vor_{\Theta(G)}(\Sigma)=\{S\in\PD_m: \Sigma_{ij}=S_{ij}\text{ for all } ij\in E(G)\text{ and } i=j\}.$$

\begin{example} Consider the graphical model associated to the undirected path $1-2-3-4$ on four vertices. This model is  
$$\Theta(G)=\{\Sigma\in\PD_4:(\Sigma^{-1})_{13}=(\Sigma^{-1})_{14}=(\Sigma^{-1})_{24}=0\}.$$
Let $\Sigma= \left(\begin{array}{rrrr}
6 & 1 & 1/7 & 1/28 \\
1 & 7 & 1 & 1/4 \\
1/7 & 1 & 8 & 2 \\
1/28 & 1/4 & 2 & 9
\end{array}\right).$ Then the logarithmic Voronoi cell at $\Sigma$ is
$$\log\Vor_{\Theta(G)}(\Sigma)=\left\{(x,y,z): M_{x,y,z} = \left(\begin{array}{rrrr}
6 & 1 & x & y \\
1 & 7 & 1 & z \\
x & 1 & 8 & 2 \\
y & z & 2 & 9
\end{array}\right)\succ 0\right\}.$$
We plot the algebraic boundary of this spectrahedron in the left figure below. It is defined by the quartic $\det(M_{x,y,z})$.
The right figure is the spectrahedron itself where ``ears'' are removed by the quadric that is the third principal minor of $M_{x,y,z}$. 
\begin{center}
\includegraphics[width=0.4\textwidth]{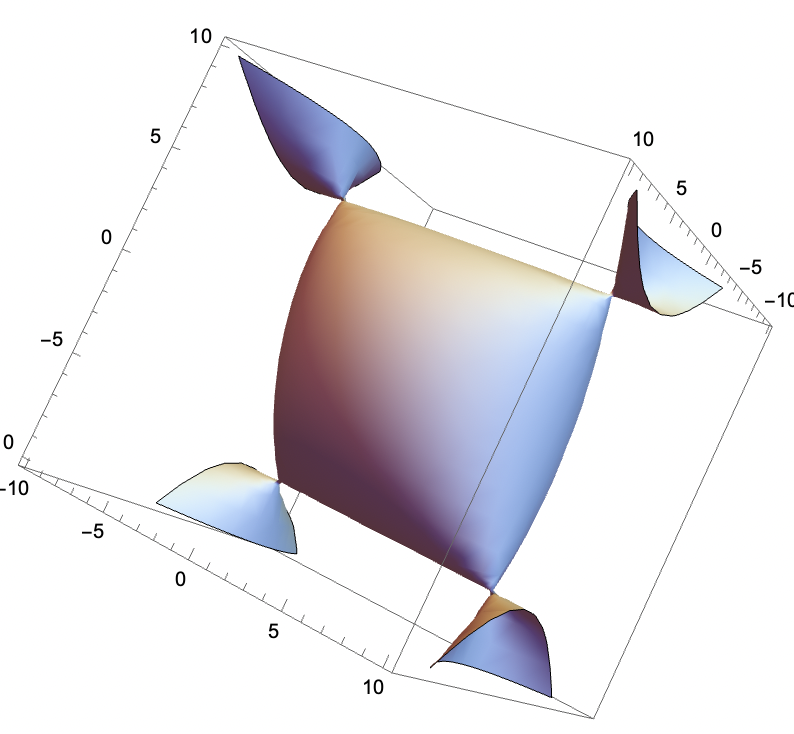},  
\includegraphics[width=0.35\textwidth]{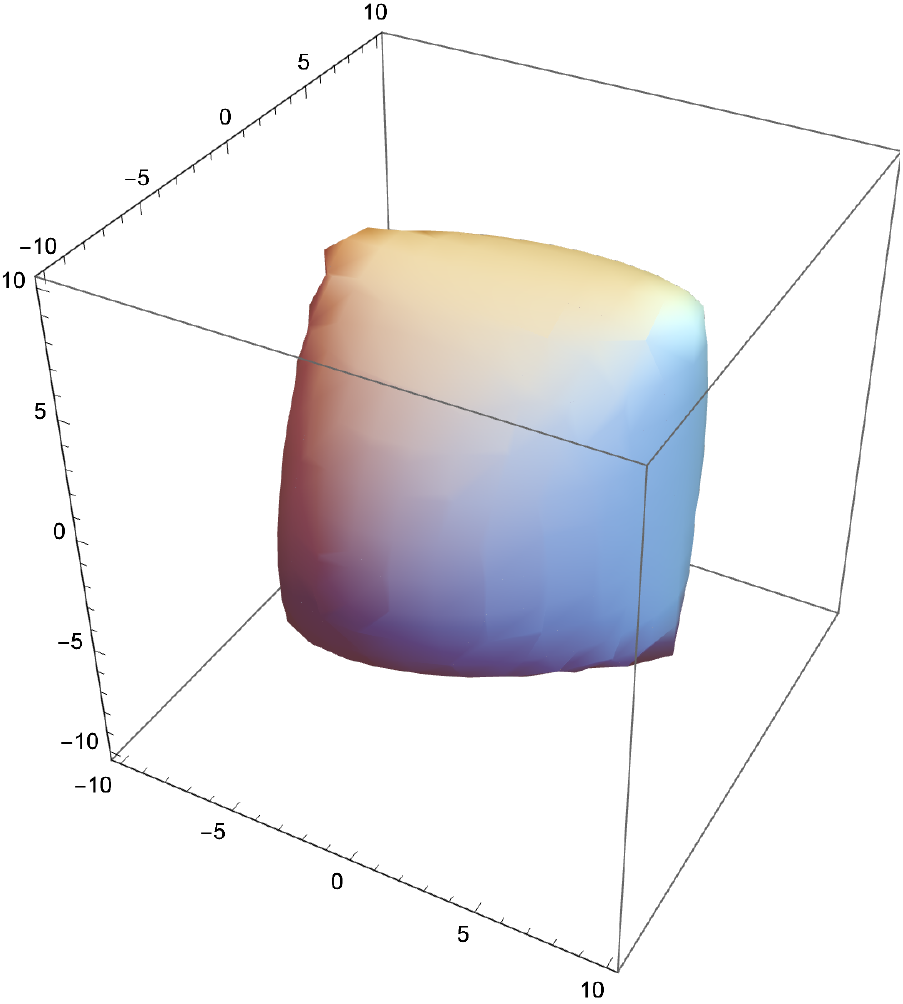}.  
\end{center}
\end{example}
The boundary of $\log\Vor_{\Theta(G)}(\Sigma)$ consists of matrices of rank at most three. This spectrahedron has four singular points that have rank two.  An interesting problem would be to study the logarithmic Voronoi cells of graphical concentration models combinatorially. In the discrete setting, it was done for linear models \cite{alexandr-linear}. In the Gaussian setting, combinatorial types of spectrahedra can be described using \textit{patches}; see \cite{patches, plaumann2021families}.

\subsection{Decomposition of logarithmic Voronoi cells}

In the theory and practice of graphical models, {\it reducible} and {\it decomposable} models
play a significant role \cite{lauritzen1996}, \cite{Sul18}. They provide a recursive structure that can be exploited, for instance, in maximum likelihood estimation. In this subsection, we develop a decomposition theory of the logarithmic Voronoi cells for such models. 

Let $G=(V,E)$ be an undirected graph with the vertex set labeled by $[m]$. A clique of $G$ is a subset $C\subseteq [m]$ such that $ij\in E(G)$ for every $i,j\in C$. We say that a clique in $G$ is \textit{maximal} if the subgraph it induces does not embed into a larger clique of $G$. Let $\C(G)$ denote the set of all cliques of $G$. Note that $\C(G)$ is a simplicial complex on $[m]$, whose facets are the maximal cliques of $G$.  

A simplicial complex $\Gamma\subseteq[m]$ is called \textit{reducible} with decomposition $(\Gamma_1,T,\Gamma_2)$ if there exist sub-complexes $\Gamma_1$, $\Gamma_2$ of $\Gamma$ and a subset $T\subseteq[m]$ such that $\Gamma=\Gamma_1\cup\Gamma_2$ and $\Gamma_1\cap\Gamma_2=2^T$. Moreover, we assume that $\Gamma_i\neq 2^T$ for $i=1,2$. We say $\Gamma$ is \textit{decomposable} if it is reducible and each of the $\Gamma_1,\Gamma_2$ is either decomposable or a simplex. A graphical model associated to an undirected graph $G$ is 
reducible (resp. decomposable) if its complex of cliques $\C(G)$ is reducible (resp. decomposable). 

Given a graph $G$ on $m$ vertices with the complex of cliques $\C(G)$, let $\Theta(G)$ denote the associated graphical model. Suppose $\Theta(G)$ is reducible with a decomposition $(\Gamma_1, T, \Gamma_2)$ of $\C(G)$. Note that the simplicial complex $\Gamma_i$ is the clique complex 
of a subgraph $G_i \subset G$ for $i=1,2$, and the intersection of $G_1$ and $G_2$
is the complete graph on the vertex set $T$.  We will denote the vertex set of $G_1$ by $U$ and the vertex set of $G_2$ by $W$. Associated to these subgraphs we have graphical models $\Theta(G_1)$ and $\Theta(G_2)$. When $A$ is an $m\times m$ matrix whose rows and columns are indexed by $[m]$, we let $A_{IJ}$ denote the submatrix of $A$, whose rows are indexed by $I \subseteq [m]$ and whose columns
are indexed by $J \subseteq [m]$. For any $|I|\times |J|$ matrix $B=(b_{ij})_{i\in I, j\in J}$, we define $[B]^{[m]}$ to be the matrix obtained from $B$ by filling in zero entries to obtain a $m\times m$ matrix, i.e.
$$([B]^{[m]})_{ij}:=\begin{cases}b_{ij}\text{ if }i \in I , j \in J\\ 0\text{ otherwise}\end{cases}.$$
The maximum likelihood estimate of $S \in \PD_m$ in $\Theta(G)$ can be computed as follows.

\begin{prop}\cite[Proposition 5.6]{lauritzen1996}
Let $\Theta(G)$ be a reducible graphical model on the undirected graph $G$ with a decomposition
$(\Gamma_1, T, \Gamma_2)$.
Let $S\in\PD_m$, and let $\hat{\Sigma}_{UU}$ be the MLE of $S_{UU}$ in $\Theta(G_1)$ and let $\hat{\Sigma}_{WW}$ be the MLE of $S_{WW}$ in $\Theta(G_2)$. Let $\hat{\Sigma}_{TT}:=S_{TT}$. 
The maximum likelihood estimate $\hat{\Sigma}$ of $S$ in $\Theta(G)$ is given by
$$\hat{\Sigma}^{-1}=[(\hat{\Sigma}_{UU})^{-1}]^{[m]}+[(\hat{\Sigma}_{WW})^{-1}]^{[m]}-[(\hat{\Sigma}_{TT})^{-1}]^{[m]}.$$
\end{prop}
For a graph $G$ and a matrix $\Sigma\in \Theta(G)$, we denote the logarithmic Voronoi cell at $\Sigma$ by $\log\Vor_{G}(\Sigma)$. For reducible graphical models, we have the following decomposition theorem.
\begin{thm} \label{thm:concentration-decomposition}
Let $\Theta(G)$ be a reducible graphical model on the undirected graph $G$ with a decomposition
$(\Gamma_1, T, \Gamma_2)$.
For any matrix $\Sigma\in \Theta(G)$, the logarithmic Voronoi cell $\log\Vor_{G}(\Sigma)$ equals
\begin{align*}
\mscriptsize{\Bigg(\left\{\Big([S_{1}^{-1}]^{[m]}+[S_{2}^{-1}]^{[m]}-[\Sigma_{TT}^{-1}]^{[m]}\Big)^{-1}: S_{1}\in\log\Vor_{G_1}(\Sigma_{UU})\text{ and } S_{2}\in\log\Vor_{G_2}(\Sigma_{WW})\right\}
+\ker(\psi)\Bigg)\cap \PD_m},
\end{align*}
where $\psi:\Sym(\RR^{m})\to\Sym(\RR^{U})\times\Sym(\RR^{W})$ is the map $$\psi:M\mapsto(M_{UU},M_{WW}).$$
\end{thm}

\begin{proof}
First, observe that the projections $\Sigma_{UU}$ and $\Sigma_{WW}$ are in $\Theta(G_1)$ and $\Theta(G_2)$, respectively. This follows from the Schur complement formula for matrix inverses. Let $A:=U \setminus T$ and $B:=W \setminus T$. First consider the matrix $S$ given by 
$$S^{-1}=[S_{1}^{-1}]^{[m]}+[S_{2}^{-1}]^{[m]}-[\Sigma_{TT}^{-1}]^{[m]}$$
where $S_{1}\in\log\Vor_{G_1}(\Sigma_{UU})\text{ and } S_{2}\in\log\Vor_{G_2}(\Sigma_{WW})$. We will show that $S\in\log\Vor_G(\Sigma)$. Recall that the logarithmic Voronoi cell at $\Sigma$ is the set 
$$\log\Vor_{G}(\Sigma)=\{S\in\PD_m: \Sigma_{ij}=S_{ij}\text{ for all } ij\in E(G)\text{ and } i = j\}.$$
Hence, it suffices to show that $S_{CC}=\Sigma_{CC}$ for every clique $C$ of $G$. Note first that we may write $S^{-1}$ in the block form as follows:
$$S^{-1}=\begin{bmatrix}(S_{1}^{-1})_{AA}&(S_{1}^{-1})_{AT}&0\\
(S_{1}^{-1})_{TA}&(S_{1}^{-1})_{TT}+(S_{2}^{-1})_{TT}-(\Sigma_{TT}^{-1})&(S_{2}^{-1})_{TB}\\
0&(S_{2}^{-1})_{BT}&(S_{2}^{-1})_{BB}
\end{bmatrix}.$$
Using Schur complements one checks that
$S_{UU}=((S^{-1})^{-1})_{UU}=(S_1^{-1})^{-1}=S_1$
and $S_{WW} = S_2$. Now, let $C$ be a clique in $G$, so either $C\subseteq\Gamma_1$ or $C\subseteq\Gamma_2$. Without loss of generality, assume $C\subseteq\Gamma_1$. Then
$$S_{CC}=(S_{UU})_{CC}=(S_{1})_{CC}=(\Sigma_{UU})_{CC}=\Sigma_{CC},$$
and we conclude that $S\in\log\Vor_G(\Sigma)$.

Now let $M=(m_{ij})\in\ker(\psi)$, i.e., $m_{ij}=0$ for all $ij \in E(G)$ or $i=j$. In particular, $M_{CC}=0$ for every clique $C$ of $G$. Thus, if $S+M$ is positive definite, we have $S+M\in\log\Vor_G(\Sigma)$, as desired.

For the other direction, let $S\in\log\Vor_{G}(\Sigma)$. Define $S_1:=S_{UU}$ and $S_2:=S_{WW}$. Note that for any clique $C\subseteq\Gamma_1$, we have $(S_1)_{CC}=(S_{UU})_{CC}=S_{CC}=\Sigma_{CC}=(\Sigma_{UU})_{CC}$, so $S_1\in\log\Vor_{G_1}(\Sigma_{UU})$. Similarly, $S_2\in\log\Vor_{G_2}(\Sigma_{WW})$. Let 
$L =[S_1^{-1}]^{[m]}+[S_2^{-1}]^{[m]}-[\Sigma_{TT}^{-1}]^{[m]}$,
and let $M:=S-(L^{-1})$. Note that $S=L^{-1}+(S-L^{-1})=L^{-1}+M$, so it suffices to show that $M\in\ker(\psi)$. We observe that 
$(L^{-1})_{UU}=((S_1^{-1})^{-1})_{UU}=S_1=S_{UU}$,
so $M_{UU}=S_{UU}-(L^{-1})_{UU}=0$. Similarly, we find that $M_{WW}=0$. Hence, indeed $M\in\ker(\psi)$, and this concludes the proof.
\end{proof}

\color{black}
\begin{remark}
The analogous decomposition theorem holds for discrete hierarchical models associated to a reducible simplicial complex. The proof is parallel to the one we presented above where 
sums of matrices are replaced by products of entries of points and differences of matrices 
are replaced by ratios of entries of points. In both cases, the decomposition of logarithmic
Voronoi cells is interesting: $\log\Vor_G(\Sigma)$ as well as $\log \Vor_{G_1}(\Sigma_{UU})$
and $\log \Vor_{G_2}(\Sigma_{WW})$ are spectrahedra, but the first term in the Minkowski
sum in Theorem \ref{thm:concentration-decomposition} is a nonlinear object. How the geometry and combinatorics of the spectrahedra $\log \Vor_{G_1}(\Sigma_{UU})$ and $\log \Vor_{G_2}(\Sigma_{WW})$ affect that of $\log \Vor_G(\Sigma)$ via this decomposition is worthwhile to study in a future project. 
\end{remark}

\section{Directed graphical models}

In this section we turn to Gaussian models defined by {\it directed acyclic graphs} (DAGs).
A DAG $G$ consists of a vertex set $V$ of cardinality $m$ and a set $E$ of directed edges $(i,j)$ without a directed cycle. We will assume that $(i,j) \in E$ implies $i < j$. Such a topological ordering of the vertices can always be achieved. For each vertex $j \in G$ there is a 
normal random variable $X_j$ such that $X_j = \sum_{k \in \mathrm{pa}(j)} \lambda_{jk} X_k + \epsilon_j$. Here $\mathrm{pa}(j)$ denotes the set of parents of the vertex $j$. The coefficients
$\lambda_{jk}$ are real parameters, known as \textit{regression coefficients}. The term $\varepsilon_j$ is a random variable that has a univariate normal distribution. This model can be summarized by the 
identity 
$$ X = \Lambda^T X + \varepsilon$$
where $\Lambda = (\lambda_{kj})$ is an upper triangular matrix with $\lambda_{kk} = 0$ for $k=1, \ldots, m$. The joint random variable $X = (X_1, \ldots, X_m)^T$ has a Gaussian distribution with covariance
matrix $\Sigma$. We also denote the diagonal covariance matrix of $\varepsilon$ by $\Omega$.
With this $\Sigma = (I - \Lambda)^{-T} \Omega (I - \Lambda)^{-1}$, and 
the maximum likelihood estimation aims to estimate the $|E|$ and $m$ parameters 
in $\Lambda$ and $\Omega$, respectively. The concentration matrix $K = \Sigma^{-1}$ is
equal to $(I - \Lambda) \Omega^{-1} (I - \Lambda)^T$. The maximum likelihood estimate can 
be found by solving a sequence of independent least squares problems for each vertex in the graph coming from the gradient of the log-likelihood function: 
Given $n$ independent observations 
of the random variable $X$, we collect them in a $n \times m$ matrix $Y$. Then the log-likelihood function is 
$$ \ell_n(\Sigma, Y^TY) = \log \det K - \tr(Y^TYK) = - \sum_{k=1}^m \log \Omega_{kk} - \sum_{k=1}^m \frac{1}{\Omega_{kk}} [(Y(I-\Lambda))^T(Y(I-\Lambda))]_{kk}.$$
One observes the solutions to $\nabla \ell_n = 0$ are obtained 
by first minimizing $|| Y_k - \sum_{j=1}^{k-1} \lambda_{jk} Y_j||^2$ for $k=1, \ldots, m$ independently. These are least squares problems with unique solutions. This development leads to the following.
\begin{thm} \cite{Wermuth}, \cite[p. 154]{lauritzen1996} \label{thm:directed-mldegree}
The maximum likelihood degree of a Gaussian model on a directed acyclic graph is one. Therefore,
the logarithmic Voronoi cell of $\Sigma$ on such a model is equal to its log-normal spectrahedron. 
\end{thm}
\begin{remark}
We were not aware of the fact that the maximum likelihood degree of Gaussian graphical models
on DAGs is one until we observed this through multiple computations. We are grateful to Piotr Zwiernik for sharing with us his notes of  the proof which we outlined above. The references we included point to the same result.
\end{remark}
For algebraic computations a convenient parametrization for Gaussian models on DAGs based on the {\it trek rule} exists \cite{seth-graphical-paper}. In this parametrization, for each directed edge $(i,j) \in E$ there is $\lambda_{ij}$ and for each vertex $i \in [m]$ there is $a_i$. For 
each pair of vertices $i,j \in [m]$, we let $T(i,j)$ be the set of paths from $i$ to $j$ which do 
not contain colliders where a collider is a pair of edges $(s,t)$ and $(u,t)$ with the same head. Such a path without colliders is called a trek. Every trek $P$ from $i$ to $j$ is a sequence of edges from $i$ up to $\mathrm{top}(P)$, the ``top''  vertex on the path, and then 
a sequence of edges down to $j$. With this the parametrization of the entries of the covariance matrix $\Sigma$ reads as follows:
$$ \sigma_{ij} \, = \, \sum_{P \in T(i,j)} a_{\mathrm{top}(P)} \prod_{(k,l) \in P} \lambda_{kl}.$$
We note that $\sigma_{ii} = a_i$, and if $T(i,j) = \emptyset$ then $\sigma_{ij} = 0$.
\begin{example}
Consider the DAG $1\to 2\to 4\leftarrow 3.$ The associated graphical model $\Theta$ is seven-dimensional. We may express $\Sigma=(\sigma_{ij})\in\Theta$ parametrically as
\begin{align*}
    &\sigma_{ii}=a_i \text{ for } i=1,2,3,4,\\
    &\sigma_{12}=a_1\lambda_{12},\; \sigma_{13}=0,\; \sigma_{14}=a_1\lambda_{12}\lambda_{24},\;,\sigma_{23}=0,\sigma_{24}=a_2\lambda_{24}, \sigma_{34}=a_3\lambda_{34}.\\
\end{align*}
The logarithmic Voronoi cell and hence  the log-normal spectrahedron of a general $\Sigma\in\Theta$ is three-dimensional, given as
$$\left\{\left(\begin{array}{rrrr}
a_{1} & a_{1} \lambda_{12} & x & y \\
a_{1} \lambda_{12} & a_{2} & z & a_{2} \lambda_{24} + \lambda_{34} z \\
x & z & a_{3} & a_{3} \lambda_{34} + \lambda_{24} z \\
y & a_{2} \lambda_{24} + \lambda_{34} z & a_{3} \lambda_{34} + \lambda_{24} z & 2 \, \lambda_{24} \lambda_{34} z + a_{4}
\end{array}\right)\succ 0: x,y,z\in\RR\right\}.$$
For the matrix $\Sigma'$ given by the parameters $$a_1=1,a_2=2,a_3=3,a_4=4,\lambda_{12}=1/2,\lambda_{24}=1,\lambda_{34}=1/2,$$
the spectrahedron $\log\Vor_{\Theta}(\Sigma')$ is the intersection of a quadric, defining a cylinder, and a quartic, defining a surface with five components. The intersection is the middle component of the quartic surface. We plot the quadric surface (left), the quartic surface (middle) and their intersection (right) in Figure \ref{figure:graphical-path}.

\begin{figure}[ht]
    \centering
    \includegraphics[width=0.3\textwidth]{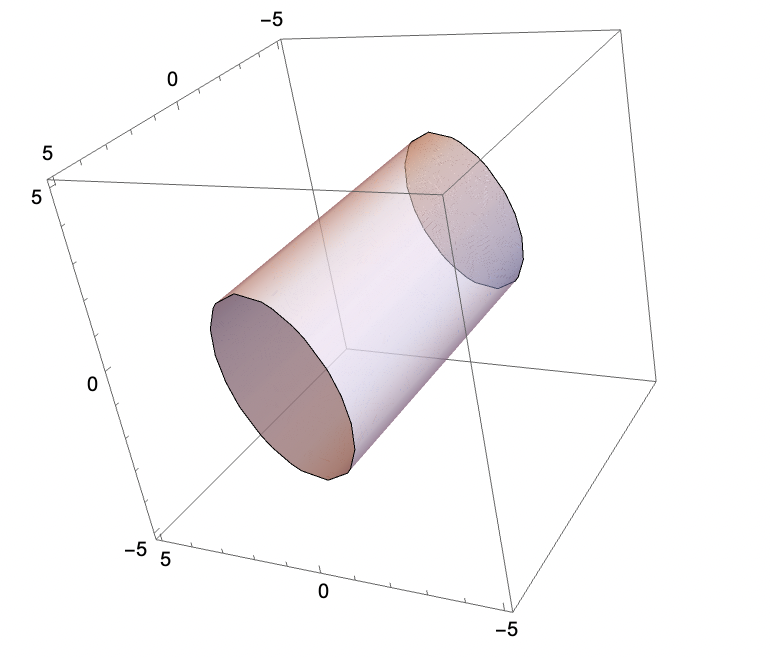} \hspace{2em}
    \includegraphics[width=0.27\textwidth]{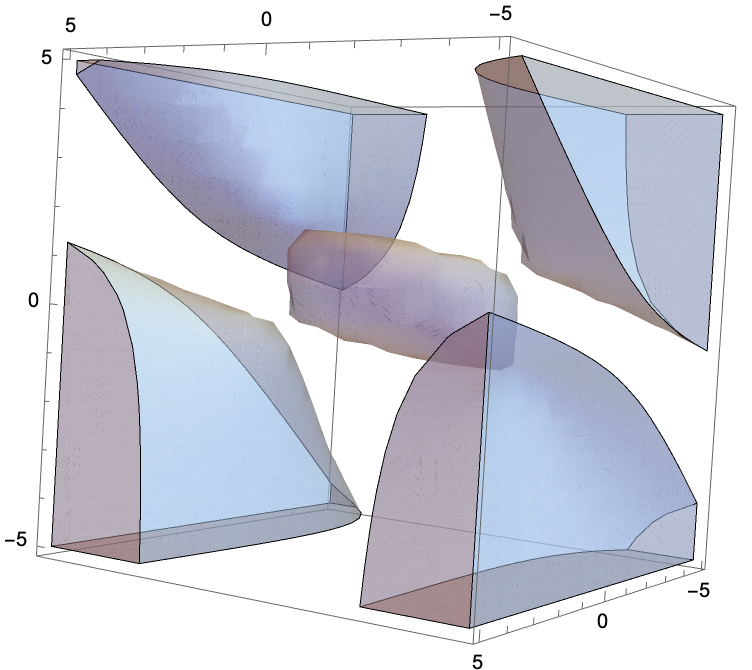} \hspace{2em}
    \includegraphics[width=0.27\textwidth]{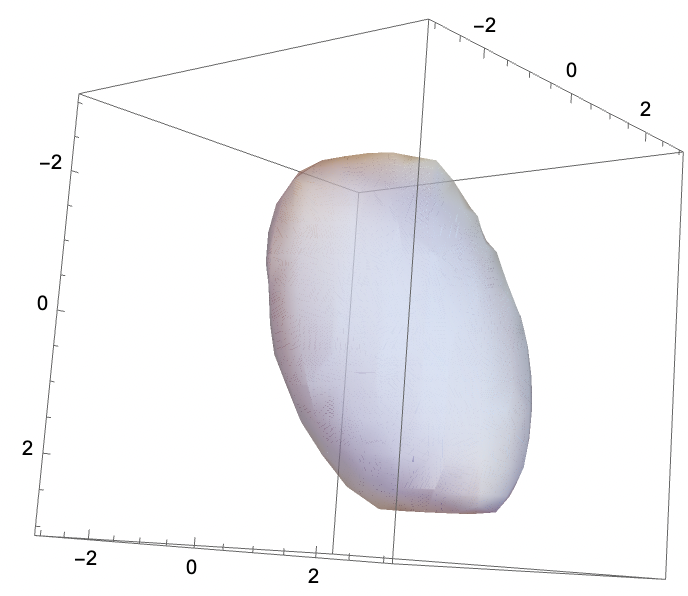} 
    \caption{The logarithmic Voronoi cell at $\Sigma'$ of $1\to 2\to 4\leftarrow 3$.}
    \label{figure:graphical-path}
\end{figure}

\end{example}

We close this section with a simple decomposition result for logarithmic Voronoi cells when the underlying graph is the disjoint union of two graphs. 
\begin{prop} \label{prop:DAG-decomposition}
Let $G$ be a DAG with vertex set $[m]$ such that $G$ is a disjoint union of two graphs $G_1$
and $G_2$ with vertex sets $U$ and $W = [m] \setminus U$, respectively. 
Then 
$$ \Theta(G) = \left\{ \Sigma = \begin{pmatrix} \Sigma_1 & 0_{UW} \\ 0_{WU} & \Sigma_2 \end{pmatrix} \, : \, \Sigma_1 \in \Theta(G_1), \Sigma_2 \in \Theta(G_2) \right\},$$
and for $\Sigma \in \Theta(G)$
$$ \log\Vor_{\Theta(G)}(\Sigma) = \left\{ \begin{pmatrix} S_1 & S_{UW} \\ S_{WU} & S_2 \end{pmatrix} \succ 0 \, : \, S_1 \in \log\Vor_{\Theta(G_1)}(\Sigma_1), S_2 \in \log\Vor_{\Theta(G_2)}(\Sigma_2)\right\}.$$
\end{prop}
\begin{proof}
The first statement is a direct consequence of Proposition 3.6 in \cite{seth-graphical-paper}. 
The second statement follows from the observation that $\ell_n(\Sigma, S) = \ell_n(\Sigma_{UU}, S_{UU}) + \ell_n(\Sigma_{WW}, S_{WW})$.
\end{proof}
\section{Covariance models}

Let $A\in\PD_m$ and let $\mathcal{L}$ be a linear subspace of $\text{Sym}(\RR^m)$. Then $A+\mathcal{L}$ is an affine subspace of $\text{Sym}(\RR^m)$. Models defined by $\Theta=(A+\mathcal{L})\cap \PD_m$ are called \textit{covariance models}. For such models, a necessary condition for $\Sigma\in\PD_m$ to be a local maximum of the log-likelihood function is $\Sigma-A\in \mathcal{L}$ and $K-KSK\in\mathcal{L}^{\perp}$ where $K = \Sigma^{-1}$; see  \cite{STZ20} and \cite{amendola2021CorrelationModels}.
From this, one can describe the log-normal spectrahedron at $\Sigma$ in the model $\Theta$ explicitly. 
\begin{prop}
The log-normal spectrahedron $\mathcal{K}_{\Sigma}(\Theta)$ at $\Sigma \in \Theta$ on a covariance model  $\Theta = (A + \mathcal{L}) \cap \PD_m$ is equal to $\N_{\Sigma}\Theta \cap \PD_m$ where
$$\N_{\Sigma}\Theta=\{S\in\text{Sym}(\RR^m): K-KSK\in\mathcal{L}^{\perp}\}.$$
\end{prop}
The log-likelihood function $\ell_n(\Sigma, S)$ is generally not concave on a covariance model, and the
maximum likelihood degree of such models can be arbitrarily high \cite{STZ20}. Therefore, in general, the logarithmic Voronoi cells are strictly contained in log-normal spectrahedra. 
On the other hand, we would like to point out the following interesting result.
\begin{prop} \cite[Proposition 3.1]{ZUR17} Let $\Theta \subseteq \PD_m$ be a Gaussian covariance model and let $S \in \PD_m$. The log-likelihood function $\ell_n(\Sigma, S)$ is 
strictly concave on the convex set $ \Delta_{2S} = \{ \Sigma \in \PD_m\, : \, 0 \prec \Sigma \prec 2S\}$, and hence it is strictly concave  on $\Delta_{2S} \cap \Theta$.
\end{prop}
This proposition immediately implies the following.
\begin{cor}
Let $\Theta \subseteq \PD_m$ be a Gaussian covariance model and let $\Sigma \in \Theta$. 
Then we have the following containments:
\begin{align}\label{two-containments}
    \{S\in\mathcal{K}_{\Theta}(\Sigma):0\prec\Sigma\prec 2S\}\subseteq\log\Vor_{\Theta}(\Sigma)\subseteq \mathcal{K}_{\Theta}(\Sigma).
\end{align}
\end{cor}

In general, both containments may be strict, as demonstrated in the next example.
\begin{example} \label{ex:elliptope}
Consider the covariance model  given by
$$\Theta=\left\{\Sigma\in\PD_3:\Sigma = \begin{pmatrix} 1 & x & z \\ x & 1 & y \\ z & y & 1 \end{pmatrix} \right\}.$$ 
This is an unrestricted correlation model. Its 
ML degree is $15$. We can represent each matrix $\Sigma\in\Theta$ by the triple $(\Sigma_{12},\Sigma_{23},\Sigma_{13})$. To see that the two containments in (\ref{two-containments}) are strict, first consider the matrix $\Sigma'=(1/2,1/3,1/4)\in\Theta$, and let $$S=\left(\begin{array}{rrr}
1211/4560 & -217/3420 & 1/30 \\
-217/3420 & 827/2565 & 1/9 \\
1/30 & 1/9 & 1
\end{array}\right)\in\mathcal{K}_{\Theta}(\Sigma').$$  
The log-likelihood function $\ell_n(\Sigma,S)$ has $15$ critical points, three of which are real. The real points are given numerically by 
$$\{(1/2,1/3,1/4),(-0.73841,0.213623,-0.0580265),(0.182141,0.316592,0.190067)\}.$$
The values of the log-likelihood function are, respectively
\begin{align*}
    &-1.53844955693696,\\
    &-1.24750351572487,\\
    &-1.55375020617405.
\end{align*}
We see that the maximum is achieved at the second point, meaning $S\notin\log\Vor_{\Theta}(\Sigma')$. This shows that the second containment in (\ref{two-containments}) is strict. To see that the first containment is strict, let
$$S=\left(\begin{array}{rrr}
813/304 & 103/76 & 1/2 \\
103/76 & 85/57 & 1/3 \\
1/2 & 1/3 & 1/3
\end{array}\right)\in\mathcal{K}_{\Theta}(\Sigma').$$
The matrix $2S-\Sigma'$ is not positive definite. However, $\ell_n(\Sigma,S)$ has only one real critical point, namely $\Sigma'$. Thus $S\in\log\Vor_{\Theta}(\Sigma')$, which shows that the first containment is also strict.
\end{example}

\subsection{Bivariate correlation model}
A \textit{bivariate correlation model} is an affine covariance model given parametrically as
$$\Theta=\left\{\Sigma_x:=\begin{pmatrix}1&x\\x&1\end{pmatrix}: x\in(-1, 1)\right\}.$$
Maximum likelihood estimation of this model has been studied extensively in \cite{amendola2021CorrelationModels}. In this section we give an explicit description of its logarithmic Voronoi cells and show that they are semialgebraic sets. This is extremely surprising. As the development below will demonstrate, the potential constraints which define 
the boundary of logarithmic Voronoi cells of these one-dimensional models are very complicated.
In particular, they are not algebraic. Nevertheless, one recovers a semi-algebraic description.

Given a sample covariance matrix $S$, the derivative of the log-likelihood function $\ell_n(\Sigma, S)$ with respect to $x$ is $\frac{2}{(1-x^2)^2}\cdot f(x)$ where
$$f(x)=x(x^2-1)-S_{12}(1+x^2)+x (S_{11}+S_{22}).$$ This polynomial has at least one real root in the interval $(-1,1)$, which corresponds to a positive definite covariance matrix in the model. This tells us that the MLE always exists, and hence the logarithmic Voronoi cells fill the cone $\PD_2$. Letting $a=(S_{11}+S_{22})/2$ and $b=S_{12}$, the polynomial $f$ can be re-written as
$f(x)=x^3-bx^2-x(1-2a)-b.$ This polynomial has either one or three real roots in the interval $(-1,1)$. In the first case, there is a unique positive definite matrix that appears as a critical point when optimizing $\ell_n(\Sigma, S)$. In the second case, there are three possible positive definite critical points. As shown in \cite{amendola2021CorrelationModels}, the latter happens if and only if $\Delta_f(b,a)>0$ and $a<1/2$, where 
$$\Delta_f(b,a)=-4[b^4-(a^4+8a-11)b^2+(2a-1)^3]$$
is the discriminant of $f$.

Fix $c\in(-1,1)$. We wish to compute the logarithmic Voronoi cell at $\Sigma_c$. Note that for a sample covariance matrix $S$ to have $\Sigma_c$ as a critical point, $c$ must be a root of $f(x)$. Substituting $c$ for $x$ in $f$, we get an equation $f(c)=0$ in $a$ and $b$. From this equation, we may express $a$ in terms of $b$ and $c$:
\begin{align}\label{relationship-between-a&b}
    a=\frac{b c^{2} - c^{3} + b + c}{2 \, c}.
\end{align}
Only $S\in\PD_2$ that satisfy this equation will have $\Sigma_c$ as a critical point when maximizing $\ell_n(\Sigma, S)$. If for such $S$ we have $\Delta_f(b,a)\leq0$ or $a\geq1/2$, then $S$ has $\Sigma_c$ as the MLE and thus $S\in\log\Vor_{\Theta}(\Sigma_c)$. If $\Delta_f(b,a)>0$ and $a<1/2$, then we must compare the value that $\ell_n$ takes on $\Sigma_c$ to the values that it takes on the two matrices $\Sigma_1,\Sigma_2$ corresponding to the other two real roots of $f(x)$, for the fixed $a$ and $b$. Given the relationship between $a$ and $b$ as in (\ref{relationship-between-a&b}) we find all three roots of $f(x)$ in terms of $b$ and $c$. They are
\begin{align*}
    c_1=\frac{b c - c^{2} - \sqrt{b^{2} c^{2} - 2 \, b c^{3} + c^{4} - 4 \, b c}}{2 \, c}\\
    c_2=\frac{b c - c^{2} + \sqrt{b^{2} c^{2} - 2 \, b c^{3} + c^{4} - 4 \, b c}}{2 \, c}
\end{align*}
and, of course, $c$ itself.

Let $$S_{b,k}=\left(\begin{array}{rr}
k & b \\
b & 2a-k
\end{array}\right)\succ 0, \;\;\; 0<k<2a $$
denote a general matrix in $\PD_2$ that has $\Sigma_c$ as a critical point when computing the MLE. In particular, the relation (\ref{relationship-between-a&b}) is satisfied. This set of matrices forms the log-normal spectrahedron $\mathcal{K}_{\Theta}(\Sigma_c)$ of $\Sigma_c$ .

\begin{theorem}\label{log}
Let $\Theta$ be the bivariate correlation model and let $\Sigma_c\in\Theta$. If $c>0$, then $\log\Vor_{\Theta}(\Sigma_c)=\{S_{b,k} \in \mathcal{K}_{\Theta}(\Sigma_c):b\geq 0\}$. If $c<0$, then $\log\Vor_{\Theta}(\Sigma_c)=\{S_{b,k}\in\mathcal{K}_{\Theta}(\Sigma_c):b\leq 0\}$. If $c=0$, then $\log\Vor_{\Theta}(\Sigma_c)=\{\diag(k,2a-k): a\geq 1/2, 0<k<2a\}$. In particular, logarithmic Voronoi cells of $\Theta$ are semi-algebraic sets.
\end{theorem}
\begin{proof}
First, suppose that $c>0$. Since we only consider the positive definite matrices $S_{b,k}$, we are working in the cone $a>|b|$. This gives us the restriction $b > c(c-1)/(c+1)$.
Note that
$$\Delta_f(b,c)=\frac{{\left(b^{2} c - 2 \, b c^{2} - 4 \, b + c^{3}\right)} {\left(b c^{2} - 2 \, c^{3} - b\right)}^{2}}{c^{3}}.$$
Thus, $\Delta_f\leq 0$ if and only if 
\begin{align}\label{ineqs-for-neg-discrim}
    \frac{c^{2} - 2 \, \sqrt{c^{2} + 1} + 2}{c}\leq b\leq \frac{c^{2} + 2 \, \sqrt{c^{2} + 1} + 2}{c}.
\end{align}
Moreover, since $a=\frac{b c^{2} - c^{3} + b + c}{2 \, c}$, we also have $a\geq 1/2$ if and only if  $b\geq \frac{c^3}{c^2+1}$. Since for $c>0$, we always have $$\frac{c^3}{c^2+1} \leq \frac{c^{2} + 2 \, \sqrt{c^{2} + 1} + 2}{c},$$
it follows that 
$$\left\{S_{b,k}\in \mathcal{K}_{\Theta}(\Sigma_c):  b\geq \frac{c^{2} - 2 \, \sqrt{c^{2} + 1} + 2}{c}\right\}\subseteq \log\Vor_{\Theta}(\Sigma_c).$$
Now, suppose $b
\leq\frac{c^{2} - 2 \, \sqrt{c^{2} + 1} + 2}{c}$. Such sample covariance matrices $S_{b,k}$ will have three positive definite roots when optimizing $\ell_n(\Sigma, S)$, namely $\Sigma_{c_1}$, $\Sigma_{c_2}$, and $\Sigma_{c}$. In order for such matrix $S_{b,k}$ to be in $\log\Vor_{\Theta}(\Sigma_c)$, it has to be the case that
$$\ell_n(\Sigma_c, S)\geq \ell_n(\Sigma_{c_i}, S)\text{ for }i=1,2.$$
A computation (in SAGE \cite{sagemath}) shows that both inequalities above are inequalities in $b$ only. The only constraints on $k$ are given by the positive definiteness of $S$. The values of the log-likelihood function are
\begin{align*}
   &\ell_n(\Sigma_{c_1}, S_{b,k})=\mtiny{-\frac{1}{D}\Bigg[{4 \, b c^{2}  + {\left(2 \, b c^{2} - c^{3} - {\left(b^{2} - 2\right)} c + 2 \, b\right)} \log\left(\frac{2 \, b c^{2} - c^{3} - {\left(b^{2} - 2\right)} c + \sqrt{b^{2} c^{2} - 2 \, b c^{3} + c^{4} - 4 \, b c} {\left(b - c\right)} + 2 \, b}{2 \, c}\right)}}\\
    &\mtiny{{+ \sqrt{b^{2} c^{2} - 2 \, b c^{3} + c^{4} - 4 \, b c} {\left({\left(b - c\right)} \log\left(\frac{2 \, b c^{2} - c^{3} - {\left(b^{2} - 2\right)} c + \sqrt{b^{2} c^{2} - 2 \, b c^{3} + c^{4} - 4 \, b c} {\left(b - c\right)} + 2 \, b}{2 \, c}\right) + 2 \, b\right)} + 2 \, b- 2 \, c^{3}- 2 \, {\left(b^{2} - 1\right)} c }\Bigg]}\\
    &\ell_n(\Sigma_{c_2}, S_{b,k})=\mtiny{-\frac{1}{D}{\Bigg[4 \, b c^{2} + {\left(2 \, b c^{2} - c^{3} - {\left(b^{2} - 2\right)} c + 2 \, b\right)} \log\left(\frac{2 \, b c^{2} - c^{3} - {\left(b^{2} - 2\right)} c - \sqrt{b^{2} c^{2} - 2 \, b c^{3} + c^{4} - 4 \, b c} {\left(b - c\right)} + 2 \, b}{2 \, c}\right)}}\\
    &\mtiny{{ - \sqrt{b^{2} c^{2} - 2 \, b c^{3} + c^{4} - 4 \, b c} {\left({\left(b - c\right)} \log\left(\frac{2 \, b c^{2} - c^{3} - {\left(b^{2} - 2\right)} c - \sqrt{b^{2} c^{2} - 2 \, b c^{3} + c^{4} - 4 \, b c} {\left(b - c\right)} + 2 \, b}{2 \, c}\right) + 2 \, b\right)} + 2 \, b - 2 \, c^{3} - 2 \, {\left(b^{2} - 1\right)} c \Bigg]}}\\
    &\ell_n(\Sigma_{c}, S_{b,k})=\mtiny{-\frac{c \log\left(-c^{2} + 1\right) + b + c}{c}},
\end{align*}
where $D=2 \, b c^{2} - c^{3} - {\left(b^{2} - 2\right)} c + \sqrt{b^{2} c^{2} - 2 \, b c^{3} + c^{4} - 4 \, b c} {\left(b - c\right)} + 2 \, b$. 
Note that for fixed $c$, the last function is linear in $b$, with the negative slope $-1/c$. At $b=0$, we always have $\ell_n(\Sigma_{c_1}, S_{0,k})=\ell_n(\Sigma_{c}, S_{0,k})>\ell_n(\Sigma_{c_2}, S_{0,k})$, so $S_{0,k}\in\log\Vor_{\Theta}(\Sigma_c)$ for $0<k<2a$. Since logarithmic Voronoi cells are convex sets by Proposition \ref{prop:logVoronoi-convex}, this means that the containment $\{S_{b,k}\in\mathcal{K}_{\Theta}(\Sigma_c): b\geq 0\}\subseteq \log\Vor_{\Theta}(\Sigma_c)$ holds. For the other containment, let $g(b)=\ell_n(\Sigma_{c_1},S_{b,k})-\ell_n(\Sigma_c,S_{b,k})$ and consider its Taylor expansion $g(b)=g(0)+g'(0)b+\cdots$. Note that $g(0)=0$ and $g'(0)<0$ for all $0<c<1$. Thus, for all $b^*<0$ with $|b^*|$ sufficiently small, the term $g'(0)b^*$ is positive and dominating in the expansion. This means that for such $b^*<0$, we have $\ell_{n}(\Sigma_{c_1}, S_{b^*,k})>\ell_{n}(\Sigma_{c}, S_{b^*,k})$, and so $S_{b^*,k}\notin\log\Vor_{\Theta}(\Sigma_c)$. Thus, $\{S_{b,k}\in\mathcal{K}_{\Theta}(\Sigma_c): b\geq 0\}= \log\Vor_{\Theta}(\Sigma_c)$ by convexity of logarithmic Voronoi cells. The proof for $c<0$ is similar. For $c=0$, we have that $b=0$ and the log-normal spectrahedron is given by $\{S_{a,k}:=\diag(k,2a-k): 0\leq k\leq 2a\}$. The two other critical points, besides $0$, are given by $c_1=\sqrt{1-2a}$ and $c_2=-\sqrt{1-2a}$. These are real if $a \leq 1/2$. In this case, the values of the log-likelihood function are as follows:
$$\ell_n(\Sigma_{c_1},S_{a,k})=\ell_n(\Sigma_{c_2},S_{a,k})=-\log(2a)-1,\;\ell_n(\Sigma_{c},S_{a,k})=-2a.$$
Note that $\ell_n(\Sigma_{c_1},S_{a,k})$ is a monotone decreasing strictly convex function and $\ell_n(\Sigma_{c},S_{a,k})$ is a linear function with slope $-2$, tangent to $\ell_n(\Sigma_{c_1},S_{a,k})$ at $a=1/2$. Thus, the only time $\Sigma_0$ is the MLE in this regime is when $a=1/2$. If $a>1/2$, $c=0$ is the 
only real critical point and gives the MLE. 
We conclude that $\log\Vor_\Theta(\Sigma_0) = 
\{\diag(k, 2a-k) \, : a\geq 1/2, \, 0 < k < 2a\}$.
\end{proof}
\begin{remark}
Note that since the bivariate correlation model is a compact set inside $\PD_2$, its log-normal spectrahedra for general matrices as well as logarithmic Voronoi cells are unbounded. In general, for $0<c<1$, the part of the log-normal spectrahedron at $\Sigma_c$ that is not in the logarithmic Voronoi cell at $\Sigma_c$ is small. This is due to the fact that $c(c-1)/(c+1)$ is a negative number with small magnitude. As $c\to 1$, the logarithmic Voronoi cell converges to the log-normal spectrahedron. In Figure \ref{bivariate-c=1/2}, we plot the logarithmic Voronoi cell at $c=1/2$ as the intersection of the pink log-normal spectrahedron and the blue half-space $b\geq 0$. Similar statement is true of $-1<c<0$.
\end{remark}

\begin{figure}[ht]
    \centering
    \includegraphics[width=0.35\textwidth]{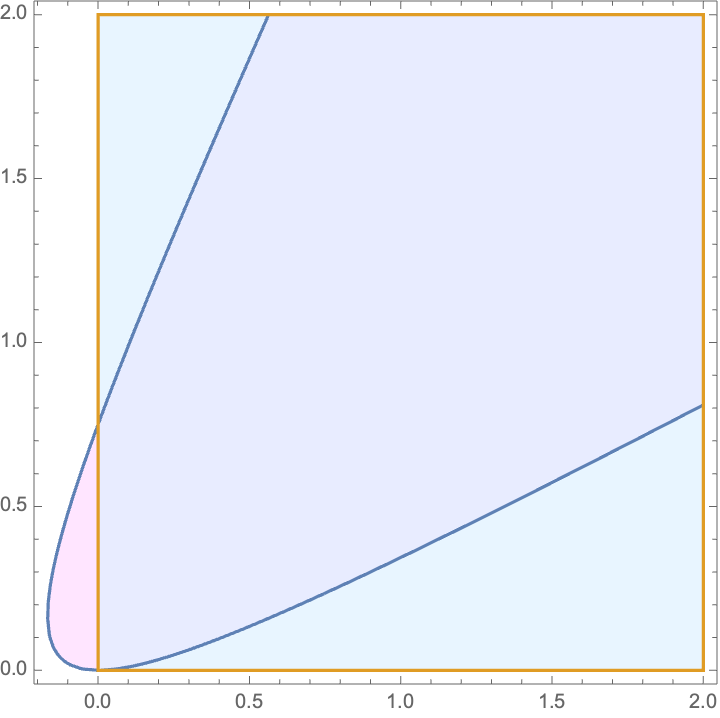} 
    \caption{The logarithmic Voronoi cell at $\Sigma_{1/2}$ for the bivariate correlation model.}
    \label{bivariate-c=1/2}
\end{figure}

\subsection{Equicorrelation models}
An \textit{equicorrelation model} is given by the parameter space
$$\Theta_m=\{\Sigma_x\in\text{Sym}(\RR^m): \Sigma_{ii}= 1, \Sigma_{ij}=x\text{ for }i\neq j, i,j\in[m], x\in\RR\}\cap \PD_m.$$
Note that this model is an instance of the affine covariance model, with $A=I_m$ and $\L=\text{span}_{\RR}\{\boldsymbol{1}\boldsymbol{1}^T-I_m\}$, where $\boldsymbol{1}$ denotes the all-ones vector in $\RR^m$. Note also that $\Theta_2$ is precisely the bivariate correlation model. For a symmetric matrix $\Sigma_x=(1-x)I_m+x\boldsymbol{1}\boldsymbol{1}^T$ to be positive definite, $-\frac{1}{m-1}<x<1$ must hold. 

Given $c\in\RR$ such that $-\frac{1}{m-1}<c<1$, we wish to describe the logarithmic Voronoi cell at $\Sigma_c\in\Theta_m$. Let $S\in\PD_m$ be a sample covariance matrix. Following \cite{amendola2021CorrelationModels}, define the \textit{symmetrized sample covariance matrix} to be the matrix
$$\bar{S}=\frac{1}{m!}\sum_{P\in S_m}PSP^T$$
where $S_m$ denotes the group of all $m\times m$ permutation matrices. Let $\mathcal{N}$ denote the space of all symmetrized sample covariance matrices. Note that for $i,j\in[m]$, we have $\bar{S}_{ii}=a$ and $\bar{S}_{ij}=b$ whenever $i\neq j$. From Lemma 5.2 in \cite{amendola2021CorrelationModels}, we have $\langle S,\Sigma_c^{-1}\rangle=\langle \bar{S},\Sigma_c^{-1}\rangle$, so optimizing $\ell_n(\Sigma,S)$ is equivalent to optimizing $\ell_n(\Sigma,\bar{S})$. Hence, we may fully recover the logarithmic Voronoi cells at $\Sigma_c$ with the matrices $\bar{S}$ for which $\ell_n(\Sigma, \bar{S})$ is maximized at $\Sigma_c$.

From Theorem 5.4 in \cite{amendola2021CorrelationModels}, we know that the ML degree of the equicorrelation model is 3 and the critical points for a general $\bar{S}$ with $\bar{S}_{ii}=a$ and $\bar{S}_{ij}=b$ for $i\neq j$ are given by the points $\Sigma_r$ where $r$ is a root of the cubic
$$f_m(x)=(m-1)x^3+((m-2)(a-1)-(m-1)b)x^2+(2a-1)x-b.$$
Since we are interested in the matrices $\bar{S}$ that have $\Sigma_c$ as a critical point, $c$ must be a root of $f_m(x)$. Then, the equation $f_m(c)=0$ becomes an equation expressing the relationship between $a$ and $b$, namely
$$a = -\frac{{\left(b + 2\right)} c^{2} - c^{3} - {\left({\left(b + 1\right)} c^{2} - c^{3}\right)} m - b - c}{c^{2} m - 2 \, c^{2} + 2 \, c}
.$$
All positive definite matrices $\bar{S}$ satisfying the above relationship are the set $\mathcal{K}_{\Theta_m}(\Sigma_c)\cap \N$. The matrices in this set depend only on the parameter $b$, so we will denote them by $\bar{S}_b$. Not all such matrices may be in the logarithmic Voronoi cell at $\Sigma_c$. The matrices $\bar{S}_b$ for which the discriminant $\Delta_{f,m}(b,a)$ is negative will be in the logarithmic Voronoi cell, since for such points, $f_m$ has only one real root. When $\Delta_{f,m}(b,a)\geq 0$
, the situation is more complicated. The good news is that most such matrices $\bar{S}_b$ satisfying $f_m(c)=0$ will still have only one positive definite critical point in the model, namely $\Sigma_c$. However, some matrices may have two additional critical points. In such cases, we have to evaluate $\ell_n(\Sigma, \bar{S}_b)$ on the other two roots of $f_m(x)$, denoted by $c_1$ and $c_2$. If $\ell_n(\Sigma_c,\bar{S}_b)$ is the largest, then $\bar{S}_b$ would be in the logarithmic Voronoi cell at $\Sigma_c$. Precisely, we have that
$$\log\Vor_{\Theta_m}(\Sigma_c)\cap\N=\{\bar{S}_b\in \mathcal{K}_{\Theta_m}(\Sigma_c)\cap \N:\ell_n(\Sigma_c,\bar{S}_b)\geq \ell_n(\Sigma_{c_i},\bar{S}_b),\; i=1,2\}.$$ For fixed $c$, the two inequalities defining the above set are inequalities in one variable $b$. Thus, the set $\log\Vor_{\Theta_m}(\Sigma_c)\cap\N$ is one-dimensional. We have the following theorem.

\begin{theorem}
Let $\Sigma_c\in \Theta_m$. The logarithmic Voronoi cell at $\Sigma_c$ is given as $$\log\Vor_{\Theta_m}(\Sigma_c)=\{S\in\PD_m: \psi(S)=\bar{S}, \bar{S}\in\mathcal{N}\cap\log\Vor_{\Theta_m}(\Sigma_c)\},$$
where 
$\psi: \PD_m\to\mathcal{N}:S\mapsto \bar{S}.$
\end{theorem}
\begin{proof}
This follows from the equality $\langle S,\Sigma_c^{-1}\rangle=\langle \bar{S},\Sigma_c^{-1}\rangle$.
\end{proof}

Note that the pre-image of any symmetrized covariance matrix $\bar{S}$ under $\psi$ has dimension $\binom{m+1}{2}-2$, and so the logarithmic Voronoi cell at any generic $\Sigma_c\in\Theta_m$ has dimension $\binom{m+1}{2}-2+1=\binom{m+1}{2}-1$, i.e. co-dimension 1, as expected.

We also observe that when $m$ increases, the number of matrices $\bar{S}_b$ that have two other positive-definite critical points besides $\Sigma_c$ decreases. Moreover, in statistical practice, such matrices $\bar{S}_b$ are rare, even for small sample sizes \cite{amendola2021CorrelationModels}. This means that for practical purposes, we may say that the logarithmic Voronoi cell at $\Sigma_c\in\Theta_m$ is \textit{approximately} its log-normal spectrahedron.
\subsection{Transcendentality of logarithmic Voronoi cells}

In this paper we have introduced logarithmic Voronoi cells for Gaussian models. In the case of models that are also well-known in algebraic statistics we have proved that the logarithmic Voronoi
cells are spectrahedra with explicit descriptions. These include linear concentration models such as Gaussian models on undirected graphs as well as Gaussian models on DAGs. The spectrahedra we have identified deserve further study. 

The case of bivariate correlation models is quite interesting since they provide the first small instance where logarithmic Voronoi cells need not be semi-algebraic. However, we showed 
that even in this case we get semialgebraicity even though the logarithmic Voronoi cells are 
not equal to the log-normal spectrahedra. The bivariate correlation models fit into a 
larger class of models known as unrestricted correlation models. 
Such a model is given by the parameter space
$$\Theta=\{\Sigma \in\text{Sym}(\RR^m): \Sigma_{ii}= 1, \, i \in [m]\}\cap \PD_m.$$
The ML degree of these models for $m \leq 6$ was computed in \cite{amendola2021CorrelationModels}. The case $m=2$ is the bivariate correlation model whose ML degree is $3$. 

When $m=3$, the model is a compact spectrahedron known as the {\it elliptope} in convex
algebraic geometry literature. We have encountered this model with ML degree $15$ in Example \ref{ex:elliptope}. The logarithmic Voronoi cells of the elliptope are unbounded $3$-dimensional convex sets. We found it quite challenging to give a good description for them besides the one
coming from its definition. We venture the following conjecture. 
\begin{conj} \label{conjecture}
The logarithmic Voronoi cells for general points on the elliptope are not semi-algebraic; in other words, their boundary is defined by transcendental functions. 
\end{conj}

\textbf{Acknowledgements}: The authors thank Carlos Am\'endola and Bernd Sturmfels for helpful discussions. This material is based upon work supported by the National Science Foundation Graduate Research Fellowship under Grants No. DGE 1752814 and DGE~2146752.

\bibliographystyle{plain}

\end{document}